\documentclass[12pt]{amsart}

\usepackage{amsmath}
\usepackage{amssymb}
\usepackage{amsfonts}
\usepackage{amsthm}
\usepackage{enumerate}
\usepackage{hyperref}
\usepackage{color}

\theoremstyle{plain}
\newtheorem{thm}{Theorem}[section]

\newtheorem{prop}[thm]{Proposition}
\newtheorem{lem}[thm]{Lemma}

\theoremstyle{definition}
\newtheorem{dfn}[thm]{Definition}

\newtheorem{dfns-rems}[thm]{Definitions and Remarks}
\newtheorem{notas-rems}[thm]{Notations and Remarks}
\newtheorem{exmps-rems}[thm]{Examples and Remarks}

\begin{document}


\title[Regularity of powers of edge ideal of whiskered cycles]{Regularity of powers of edge ideal of whiskered cycles}


\author[M. Moghimian]{M. Moghimian}

\address{M. Moghimian, Science and Research Branch, Islamic Azad University
(IAU), Tehran, Iran.}

\email{math\_moghimi@yahoo.com}

\author[S. A. Seyed Fakhari]{S. A. Seyed Fakhari}

\address{S. A. Seyed Fakhari, School of Mathematics, Institute for Research
in Fundamental Sciences (IPM), P.O. Box 19395-5746, Tehran, Iran.}

\email{fakhari@ipm.ir}

\urladdr{http://math.ipm.ac.ir/$\sim$fakhari/}

\author[S. Yassemi]{S. Yassemi}

\address{S. Yassemi, School of Mathematics, Statistics and Computer Science,
College of Science, University of Tehran, Tehran, Iran, and School of
Mathematics, Institute for Research in Fundamental Sciences (IPM), P.O. Box
19395-5746, Tehran, Iran.}

\email{yassemi@ipm.ir}

\urladdr{http://math.ipm.ac.ir/$\sim$yassemi/}


\begin{abstract}
Let $G=W(C_{n})$ be a whiskered cycle graph with edge ideal $I=I(G)$. We prove that for every $s\geq 1$, the equality ${\rm reg}(I^{s})=2s+\lceil \frac{n-1}{2}\rceil-1$ holds.
\end{abstract}






\thanks{}


\maketitle


\section{Introduction and Preliminaries} \label{sec1}

Let $I$ be a homogeneous ideal in the polynomial ring $R = \mathbb{K}[x_1,\ldots,x_n]$. Suppose that the minimal free resolution of $I$ is given by
$$0  \longrightarrow \cdots \longrightarrow  \bigoplus_{j}R(-j)^{\beta_{1,j}(I)} \longrightarrow \bigoplus_{j}R(-j)^{\beta_{0,j}(I)}   \longrightarrow  I \longrightarrow 0.$$
The Castelnuovo-Mumford regularity (or simply, regularity) of $I$, denote by ${\rm reg}(I)$, is defined as follows:
$${\rm reg}(I)=\max\{j-i|\ \beta_{i,j}(I)\neq0\}.$$
The regularity of $I$ is an important invariant in commutative algebra.

There is a natural correspondence between quadratic squarefree monomial ideals of $R$ and finite simple graphs with $n$ vertices. To every simple graph $G$ with vertex set $V(G)=\{x_1,\ldots,x_n\}$ and edge set $E(G)$, we associate an ideal $I=I(G)$ defined by
$$I(G)=\big(x_{i}x_{j}:\{x_{i}, x_{j}\}\in E(G)\big)\subseteq R.$$
Computing and finding bounds for the regularity of edge ideals and their powers have been studied by a number of researchers (see for example\cite{ab}, \cite{b}, \cite{bht}, \cite{c}, \cite{fmo}, \cite{J} and \cite{mv}).
It is well-known that ${\rm reg}(I^{s})$ is asymptotically a linear function for $s\gg0$. When $I=I(G)$ is an edge ideal, then there exist integers $b$ and $s_{0}$ such that for all $s\geq s_{0}$, ${\rm reg}(I^{s})=2s+b.$
The simplest case is when $b=0$, i.e ${\rm reg}(I^{s})=2s$, in this case, $I^{s}$ has linear minimal free resolution. For example, Banerjee \cite{b} proves that if $G$ is a gap-free and cricket-free gragh then ${\rm reg}(I(G)^{s})=2s$.
Using \cite[Theorem 3.2]{hhz} together with Fr\"{o}berg's result \cite{f}, we know that if the complement graph of $G$ is chordal, then $I(G)^{s}$ has linear resolution for all $s\geq1$.

Recently, Beyarslam, H${\rm \grave{a}}$ and Trung \cite{bht} proved that for every graph $G$ and every integer $s\geq 1$, the inequality$${\rm reg}(I(G)^{s})\geq 2s+{\rm indmatch}(G)-1$$holds, where indmatch$(G)$ denotes the induced matching number of $G$, and it is the maximum cardinality of the induced matching of $G$ (see \cite[Theorem 4.5]{bht}). In the same paper, the authors  proved the equality for every $s\geq 1$, if $G$ is a forest and for every $s\geq 2$, if $G$ is a cycle (see \cite[Theorems 4.7 and 5.2]{bht}). In this paper, we determine a  new class of graphs for which the equality$${\rm reg}(I(G)^{s})= 2s+{\rm indmatch}(G)-1$$holds, for every $s\geq 1$.

Let $G$ be a graph. Adding a whisker to $G$ at a vertex $v$ means adding a new vertex $u$ and the edge $\{u,v\}$ to $G$. The graph which is obtained from $G$ by adding a whisker to each vertex of $G$ is denoted by $W(G)$. For every integer $n\geq 3$, let $C_{n}$ be the $n$-cycle graph and set $G=W(C_{n})$. It immediately follows from \cite[Theorem 13]{bv} that ${\rm reg}(I(G))=\lceil \frac{n-1}{2}\rceil+1$ and one can easily check that ${\rm indmath}(G)=\lceil \frac{n-1}{2}\rceil$. Thus, we have the following result.

\begin{prop} \label{first}
For every integer $n\geq 3$, let $C_{n}$ be the $n$-cycle graph and set $G=W(C_{n})$. Then ${\rm reg}(I(G))=\lceil \frac{n-1}{2}\rceil+1={\rm indmath}(G)+1$
\end{prop}

As the main result of this paper, we extend Proposition \ref{first} by proving than for every $s\geq 1$, ${\rm reg}(I(W(C_n))^{s})=2s+\lceil \frac{n-1}{2}\rceil-1$ (see Theorem \ref{mmmain}).

We first need to recall some basic definitions from graph theory.

If two vertices of a graph $G$ are joined by an edge then these vertices are called neighbors. The set of all neighbors of a vertex $v$ is called the neighborhood set of $v$ and is denoted by $N(v)$. Moreover, we set $N[v]=N(v)\cup\{v\}$. The degree of a vertex $v$ of $G$ is the number of its neighbors and is denoted by ${\rm deg}_{G}(v)$.

A subgraph $H$ of $G$ is called induced provided that two vertices of $H$ are adjacent if and only if they are adjacent in $G$. For an edge $e$ in $G$, let $G\setminus e$ be the subgraph of $G$ obtained by removing the edge $e$ from $G$. For a subset $S\subseteq V(G)$ of vertices, let $G\setminus S$ be the induced subgraph on $V(G)\setminus S$ and we write $G\setminus v$ instead of $G\setminus \{v\}$.
 For an edge $e=\{u,v\}$ of $G$, we set $N_{G}[e]=N_{G}[u]\cup N_{G}[v]$ and define $G_{e}$ to be the induced subgraph $G\setminus N_{G}[e]$ of $G$.

\textbf{Convention}: Throughout this paper, for simplicity, we use ${\rm reg}(G)$ instead of ${\rm reg}(I(G))$.

We need the following result due to H${\rm \grave{a}}$ \cite{h} on the regularity of edge ideals.

\begin{thm}\label{change}
Let $G=(V(G),E(G))$ be a graph.
\begin{itemize}
\item[(i)] \cite[Lemma 3.1]{h} For any induced subgraph $H$ of $G$ we have ${\rm reg}(H)\leq {\rm reg}(G)$.

\item[(ii)] \cite[Theorem 3.4]{h} Let $x\in V(G)$. Then$${\rm reg}(G)\leq \max\{{\rm reg}(G\setminus x), {\rm reg}(G\setminus N[x])+1\}.$$

\item[(iii)] \cite[Theorem 3.5]{h} Let $e\in E(G)$. Then$${\rm reg}(G)\leq \max\{2, {\rm reg}(G\setminus e), {\rm reg}(G_{e})+1\}.$$
\end{itemize}
\end{thm}

Our method for proving the main result is based on the recent work of Banerjee \cite{b}.
We recall the following definition and theorem from \cite{b}.

\begin{dfn}
Let $G$ be a graph. Two vertices $u$ and $v$ ($u$ may be equal to $v$) are said to be even-connected with respect to an $s$-fold product $e_{1} \ldots e_{s}$ of edges of $G$, if there is a path $p_{0}, p_{1}, \ldots, p_{2l+1}$, $l\geq 1$ in $G$ such that the following conditions hold:
\begin{itemize}
\item[(i)] $p_{0}=u$ and $p_{2l+1}=v$;

\item[(ii)] for all $0\leq k\leq l-1, \{p_{2k+1},p_{2k+2}\}=e_{i}$ for some \emph{i}; and

\item[(iii)] for all \emph{i}, $\mid \{k\mid \{p_{2k+1},p_{2k+2}\}=e_{i}\}\mid \leq\mid\{j\mid e_{i}=e_{j}\}\mid$.
\end{itemize}
\end{dfn}

\begin{thm}\cite[ Theorems 6.1 and 6.7]{b} \label{increase}

Assume that $s\geq 1$ is an integer, $G$ is a graph and $I=I(G)$ is its edge ideal. Let $M$ be a minimal generator of $I^{s}$. Then the ideal $(I^{s+1}:M)$ is generated by monomials of degree two and for every generator $uv$ ($u$ may be equal to $v$) of this ideal, either $\{u,v\}$ is an edge of $G$ or $u$ and $v$ are even-connected with respect to $M$.
\end{thm}

\section{Main results} \label{sec2}

The aim of this section is to prove that for every $s\geq 1$, we have$${\rm reg}(I^{s})=2s+\lceil\frac{n-1}{2}\rceil-1,$$where $I$ is the edge ideal of a whiskered cycle $W(C_{n})$ (see Theorem \ref{mmmain}). We first need the following technical lemma. The proof of this lemma is long and we postpone it to the next section.

\begin{lem}\label{choose}
Let $G$ be a graph with $2n$ vertices, say $\{x_{1},x_{2}, \ldots,x_{2n}\}$. Assume that
\begin{itemize}

\item[(i)] $\{x_{i},x_{i+1}\}\in E(G)$ for every $1\leq i\leq n-1$.

\item[(ii)] $\{x_{i},x_{n+i}\}\in E(G)$ for every $1\leq i\leq n$.

\item[(iii)] For $2\leq i,j\leq n-1$, if $x_{n+i}$ is adjacent to $x_{n+j}$ in $G$, then $\{x_{n+i},x_{j-1}\}$, $\{x_{n+i},x_{j+1}\}$, $\{x_{n+j},x_{i-1}\}$ and $\{x_{n+j},x_{i+1}\}$ are edges of $G$.

\item[(iv)] For $2\leq i\leq n-1$, if $x_{n+i}$ is adjacent to $x_{n+1}$ in $G$, then $\{x_{n+i},x_{2}\}$, $\{x_{n+1},x_{i-1}\}$ and $\{x_{n+1},x_{i+1}\}$ are edges of $G$.

\item[(v)] For $2\leq i\leq n-1$, if $x_{n+i}$ is adjacent to $x_{2n}$ in $G$, then $\{x_{n+i},x_{n-1}\}$, $\{x_{i-1},x_{2n}\}$ and $\{x_{i+1},x_{2n}\}$ are edges of $G$; and

\item[(vi)] If $x_{n+1}$ and $x_{2n}$ are adjacent in $G$, then $\{x_{n+1},x_{n-1}\}$ and $\{x_{2},x_{2n}\}$ are edges of $G$.
\end{itemize}
Then $${\rm reg}(G)\leq \lceil\frac{n+2}{2}\rceil.$$
\end{lem}

The following lemma has a crucial role in the proof of our main result.

\begin{lem}\label{main}
Let $n\geq 3$ be an integer and $G$ be a graph with $2n$ vertices, say $\{x_{1},x_{2}, \ldots, x_{2n}\}$. Assume that
\begin{itemize}

\item[(i)] $\big\{\{x_{1},x_{2}\}, \{x_{2},x_{3}\}, \ldots, \{x_{n-1},x_{n}\} ,\{x_{n},x_{1}\}\big\}\subseteq E(G)$.

\item[(ii)] $\{x_{i},x_{n+i}\}\in E(G)$, for every $1\leq i\leq n$.

\item[(iii)] For $1\leq i,j\leq n$, if $x_{n+i}$ is adjacent to $x_{n+j}$ in $G$, then $\{x_{n+i}, x_{j-1}\}$, $\{x_{n+i}, x_{j+1}\}$, $\{x_{n+j}, x_{i-1}\}$ and $\{x_{n+j}, x_{i+1}\}$ are edges of $G$, where we consider the indices $i-1$, $i+1$, $j-1$ and $j+1$ modulo $n$.
\end{itemize}
Then $${\rm reg}(G)\leq \lceil\frac{n+1}{2}\rceil.$$
\end{lem}

\begin{proof}
We use induction on the number of vertices of $G$.

If $n=3$, then one can check that $G^{c}$ (the complement graph of $G$) is a chordal graph and thus $\rm reg(G)=2$.
If $n=4$, then $G$ contains a whiskered path of length 4, which satisfies the assumption of lemma \ref{choose} and thus its regularity is at most 3.
Now assume that $n\geq5$. By assumption $\mid E(G)\mid\geq2n$.
If $\mid E(G)\mid=2n$, then $G$ is a whiskered cycle graph and it follows from Proposition \ref{first} that$${\rm reg}(G)=\lceil \frac{n-1}{2}\rceil+1=\lceil\frac{n+1}{2}\rceil.$$
Suppose now that $\mid E(G)\mid\geq2n+1$. By induction on the number of edges of $G$, we prove that ${\rm reg}(G)\leq\lceil\frac{n+1}{2}\rceil$. We consider the following cases.

\vspace{0.2cm}
{\bf Case 1.}
Assume that there exist $1\leq i<j\leq n$ such that $j-i$ is not congruent to one modulo $n$ and $e=\{x_{i},x_{j}\}\in E(G)$. Using induction hypothesis on the number of edges, it follows that$${\rm reg}(G\setminus e)\leq\lceil\frac{n+1}{2}\rceil.$$
By Theorem \ref{change}, it is enough to show that$${\rm reg}(G_{e})\leq\lceil\frac{n-1}{2}\rceil.$$Let $H$ be the induced subgraph of $G$ over the vertices $V(G)\setminus \{x_{i},x_{n+i},x_{j},x_{n+j}\}$ and let $H'$ be the graph with the same vertex set as $H$ and the edge set
\begin{align*}
E(H') & =E(H)\cup\big\{\{x_{i-1},x_{i+1}\},\{x_{j-1},x_{j+1}\}\big\}\cup\big\{\{x_{i-1},x_{n+k}\}\mid 1\leq k\leq n, k\neq i,j\big\}\\ & \cup\big\{\{x_{i+1},x_{n+k}\}\mid 1\leq k\leq n, k\neq i,j\big\}\cup
\big\{\{x_{j-1},x_{n+l}\}\mid 1\leq l\leq n, l\neq i,j\big\}\\ & \cup\big\{\{x_{j+1},x_{n+l}\}\mid 1\leq l\leq n, l\neq i,j\big\},
\end{align*}
where we consider the indices $i-1, i+1, j-1$ and $j+1$ modulo $n$.

One can easily check that $H'$ satisfies the assumptions of the lemma. Thus, by induction on the number of vertices, $${\rm reg}(H')\leq\lceil\frac{(n-2)+1}{2}\rceil=\lceil\frac{n-1}{2}\rceil.$$
Moreover, $G_{e}$ is an induced subgraph of $H'$ which implies that $${\rm reg}(G_{e})\leq {\rm reg}(H')\leq\lceil\frac{n-1}{2}\rceil.$$

\vspace{0.2cm}
{\bf Case 2.}
Assume that there exist $1\leq i<j\leq n$ such that $j-i$ is not congruent to one modulo $n$ and  $e=\{x_{n+i},x_{n+j}\}\in E(G)$. By induction hypothesis on the number of edges, we have$${\rm reg}(G\setminus e)\leq\lceil\frac{n+1}{2}\rceil.$$Let $H$ be the induced subgraph of $G$ over the vertices $V(G)\setminus \{x_{i},x_{n+i},x_{j},x_{n+j}\}$ and let $H'$ be the graph with the same vertex set as $H$ and the edge set
\begin{align*}
E(H') & =E(H)\cup\big\{\{x_{i-1},x_{i+1}\},\{x_{j-1},x_{j+1}\}\big\}\cup\big\{\{x_{i-1},x_{n+k}\}\mid 1\leq k\leq n, k\neq i,j\big\}\\ & \cup\big\{\{x_{i+1},x_{n+k}\}\mid 1\leq k\leq n, k\neq i,j\big\}\cup
\big\{\{x_{j-1},x_{n+l}\}\mid 1\leq l\leq n, l\neq i,j\big\}\\ & \cup\big\{\{x_{j+1},x_{n+l}\}\mid 1\leq l\leq n, l\neq i,j\big\},
\end{align*}
where we consider the indices $i-1, i+1, j-1$ and $j+1$ modulo $n$.

One can easily check that $H'$ satisfies the assumptions of the lemma. On the other hand, it follows from the assumptions $\{x_{n+i},x_{j-1}\}$, $\{x_{n+i},x_{j+1}\}$, $\{x_{n+j},x_{i-1}\}$ and $\{x_{n+j},x_{i+1}\}$ are edges of $G$. Thus, $G_{e}$ is an induced subgraph of $H'$. Hence, the induction hypothesis on $n$ and Theorem \ref{change} show that $${\rm reg}(G_{e})\leq {\rm reg}(H')\leq\lceil\frac{(n-2)+1}{2}\rceil=\lceil\frac{n-1}{2}\rceil.$$
The conclusion that ${\rm reg}(G)\leq\lceil\frac{n+1}{2}\rceil$, follows from Theorem \ref{change}.

\vspace{0.2cm}
{\bf Case 3.}
Assume that there exists $1\leq i\leq n$ such that $e=\{x_{n+i},x_{n+i+1}\}\in E(G)$, (where by $x_{2n+1}$, we mean $x_{n+1}$). Then, by induction hypothesis on the number of edges,$${\rm reg}(G\setminus e)\leq\lceil\frac{n+1}{2}\rceil.$$Let $H$ be the induced subgraph of $G$ over the vertices $V(G)\setminus \{x_{i},x_{n+i},x_{i+1},x_{n+i+1}\}$ and let $H'$ be the graph with the same vertex set as $H$ and the edge set
\begin{align*}
E(H') & =E(H)\cup\big\{\{x_{i-1},x_{i+2}\}\big\}\cup\big\{\{x_{i-1},x_{n+k}\}\mid 1\leq k\leq n, k\neq i,i+1\big\}\\ & \cup\big\{\{x_{i+2},x_{n+k}\}\mid 1\leq k\leq n, k\neq i,i+1\big\},
\end{align*}
where we consider the indices $i-1$ and $i+2$ modulo $n$.

One can easily check that $H'$ satisfies the assumptions of the lemma. Thus, by induction hypothesis on $n$, $${\rm reg}(H')\leq\lceil\frac{(n-2)+1}{2}\rceil=\lceil\frac{n-1}{2}\rceil.$$
By assumptions, $\{x_{n+i},x_{i+2}\}$ and $\{x_{n+i+1},x_{i-1}\}$ are edges of $G$. Hence, $G_{e}$ is an induced subgraph of $H'$. Therefore, by Theorem \ref{change}, $${\rm reg}(G_{e})\leq {\rm reg}(H')\leq\lceil\frac{n-1}{2}\rceil,$$ and then again by Theorem \ref{change}, we conclude that ${\rm reg}(G)\leq\lceil\frac{n+1}{2}\rceil.$

\vspace{0.2cm}
{\bf Case 4.}
Assume that there exist $1\leq i<j\leq n$ such that $j-i$ is not congruent to one modulo $n$ and  $e=\{x_{n+i},x_{j}\}\in E(G)$. By cases 2 and 3, we may assume that for every $1\leq k,l\leq n$ the vertices $x_{n+k}$ and $x_{n+l}$ are not adjacent in $G$. Then, by induction hypothesis on the number of edges,$${\rm reg}(G\setminus e)\leq\lceil\frac{n+1}{2}\rceil.$$Let $H$ be the induced subgraph of $G$ over the vertices $V(G)\setminus \{x_{i},x_{n+i},x_{j},x_{n+j}\}$ and let $H'$ be the graph with the same vertex set as $H$ and the edge set
\begin{align*}
E(H')=E(H)\cup\big\{\{x_{i-1},x_{j-1}\}, \{x_{i+1},x_{j+1}\}\big\},
\end{align*}
where we consider the indices $i-1$, $i+1$, $j-1$ and $j+1$ modulo $n$. By considering the cycle$$x_1, x_2, \ldots, x_{i-1}, x_{j-1}, x_{j-2}, \ldots, x_{i+1}, x_{j+1}, x_{j+2}, \ldots, x_n, x_1$$one easily check that $H'$ satisfies the assumptions of the lemma. On the other hand, $G_{e}$ is an induced subgraph of $H'$. Thus, by induction hypothesis on the number of vertices, we have $${\rm reg}(G_{e})\leq {\rm reg}(H')\leq\lceil\frac{(n-2)+1}{2}\rceil=\lceil\frac{n-1}{2}\rceil.$$
It follows from Theorem \ref{change} that ${\rm reg}(G)\leq\lceil\frac{n+1}{2}\rceil$.

\vspace{0.2cm}
{\bf Case 5.}
Assume that there exists $1\leq i\leq n$ such that $e=\{x_{n+i},x_{i+1}\}\in E(G)$, (where we consider the index $i+1$ modulo $n$). By cases 2 and 3, we may assume that for every $1\leq k,l\leq n$,$$\{x_{n+k},x_{n+l}\}\notin E(G).$$Then, by induction hypothesis on the number of edges,$${\rm reg}(G\setminus e)\leq\lceil\frac{n+1}{2}\rceil.$$Let $H$ be the induced subgraph of $G$ over the vertices $V(G)\setminus \{x_{i},x_{n+i},x_{i+1},x_{n+i+1}\}$ and let $H'$ be the graph with the same vertex set as $H$ and the edge set
\begin{align*}
E(H')=E(H)\cup\big\{\{x_{i-1},x_{i+2}\}\big\},
\end{align*}
where we consider the indices $i-1$ and $i+2$ modulo $n$. One easily check that $H'$ satisfies the assumptions of the lemma. On the other hand, $G_{e}$ is an induced subgraph of $H'$ and by induction hypothesis on $n$, $${\rm reg}(G_{e})\leq {\rm reg}(H')\leq\lceil\frac{(n-2)+1}{2}\rceil=\lceil\frac{n-1}{2}\rceil.$$
Therefore, by Theorem \ref{change}, ${\rm reg}(G)\leq\lceil\frac{n+1}{2}\rceil$.
\end{proof}

To prove the main result, we also need the following two lemmas.

\begin{lem}\label{strong}
Let $G=W(C_{n})$ be a whiskered cycle graph with edge ideal $I=I(G)$ and assume that the vertices of the $n$-cycle (in order) are $x_{1}, \ldots, x_{n}$ and the whiskers are the edges $\{x_1, x_{n+1}\}, \ldots, \{x_n, x_{2n}\}$. Then for a minimal generator $M$ of $I^{s}$, we have $x_{n}^{2}\in(I^{s+1}:M)$ if and only if $n$ is odd, say $n=2m+1$ for some $1\leq m\leq s$ and $$M=(x_{1}x_{2})(x_{3}x_{4}) \ldots (x_{2m-1}x_{2m})N,$$ for some $N\in I^{s-m}$.

Moreover, in this case $x_{n}x_{j}\in(I^{s+1}:M)$ for all $j=1, \ldots, 2n$.
\end{lem}

\begin{proof}
First assume that $n=2m+1$ is odd and $M=(x_{1}x_{2})(x_{3}x_{4}) \ldots (x_{2m-1}x_{2m})N$ for some $N\in I^{s-m}$. Then $x_{n}^{2}M=(x_{n}x_{1})(x_{2}x_{3}) \ldots (x_{2m}x_{n})N\in I^{s+1}$.

Now, assume that $x_{n}^{2}\in(I^{s+1}:M)$. Using an argument similar to the proof of \cite[ Lemma 5.1]{bht}, one concludes that $n$ is odd, say $n=2m+1$ for some $1\leq m\leq s$ and $$M=(x_{1}x_{2})(x_{3}x_{4}) \ldots (x_{2m-1}x_{2m})N,$$ for some $N\in I^{s-m}$.

By Theorem \ref{increase}, to prove the last statement of the lemma, it is enough to show that for every $j=1, \ldots, 2n$, the vertices $x_{n}$ and $x_{j}$ are even-connected with respect to $M$. By assumptions, there is nothing to prove if $j=n$ or $2n$. Hence assume that $j\neq n ,2n$. We consider the following cases.

\vspace{0.2cm}
{\bf Case 1.}
If $1\leq j< n$ is odd, then $x_{n}, x_1, \ldots, x_{j}$ is an even-connected path between $x_{n}$ and $x_{j}$ with respect to $M$ (remember that $M=(x_{1}x_{2})(x_{3}x_{4}) \ldots (x_{2m-1}x_{2m})N,$ for some $N\in I^{s-m}$).

\vspace{0.2cm}
{\bf Case 2.}
If $1\leq j< n$ is even, then $x_{j}, x_{j+1}, \ldots, x_{n}$ is an even-connected path between $x_{j}$ and $x_{n}$ with respect to $M$ (remember that $n$ is odd).

\vspace{0.2cm}
{\bf Case 3.}
Assume that $n+1\leq j<2n$. If $j$ is even, then $x_{n}, x_{n-1}, \ldots, x_{j-n},x_{j}$ is an even-connected path between $x_{n}$ and $x_{j}$ with respect to $M$. If $j$ is odd, then $x_{n}, x_1, \ldots, x_{j-n}, x_{j}$ is an even-connected path between $x_{n}$ and $x_{j}$ with respect to $M$.

Thus, $x_{n}x_{j}\in(I^{s+1}:M)$ for all $j=1, \ldots, 2n$ and the lemma is proved.
\end{proof}

\begin{lem}\label{stable}
With the same assumptions as Lemma \ref{strong}, $x_{n+1}^{2}\in(I^{s+1}:M)$ if and only if  $n$ is odd, say $n=2m-1$ for some $2\leq m\leq s$ and $$M=(x_{1}x_{2})(x_{3}x_{4}) \ldots (x_{2m-3}x_{2m-2})(x_{2m-1}x_{1})N,$$ for some $N\in I^{s-m}$.

Moreover, in this case $x_{n+1}x_{j}\in(I^{s+1}:M)$ for all $j=1, \ldots, 2n$.
\end{lem}

\begin{proof}
First suppose that $n=2m-1$ is odd and$$M=(x_{1}x_{2})(x_{3}x_{4}) \ldots (x_{2m-3}x_{2m-2})(x_{2m-1}x_{1})N$$for some $N\in I^{s-m}$. Then $$x_{n+1}^{2}M=(x_{n+1}x_{1})(x_{2}x_{3}) \ldots (x_{2m-2}x_{2m-1})(x_{1}x_{n+1})N\in I^{s+1}.$$

Now, assume $x_{n+1}^{2}\in(I^{s+1}:M)$. Then Theorem \ref{increase} implies that $x_{n+1}$ is even-connected to itself with respect to $M$. Let $x_{n+1}=p_{0},p_{1}, \ldots, p_{2m+1}=x_{n+1}$ be a shortest even-connected path between $x_{n+1}$ and itself. Assume that there exists $1\leq j\leq 2k$ such that $p_{j}=x_{n+1}$. If $j$ is odd then $x_{n+1}=p_{0}, \ldots, p_{j}=x_{n+1}$ is a shorter even-connected path between $x_{n+1}$ and itself, a contradiction. If $j$ is even then $x_{n+1}=p_{j}, \ldots, p_{2m+1}=x_{n+1}$ is again a shorter even-connected path between $x_{n+1}$ and itself, a contradiction. Thus, we assume that $x_{n+1}$ does not appear in the path $p_{0}, \ldots, p_{2m+1}$ except at its endpoints.

We note that since ${\rm deg}_{G}(x_{n+1})=1$, then the even-connected path $p_{0}, \ldots, p_{2m+1}$ can not be simple. Hence, there exist indices $i$ and $j$ with $1\leq i<j\leq 2k$ such that $p_{i}=p_{j}$ and we choose $i$ and $j$ such that $j-i$ is minimal. Then $p_{i}, \ldots, p_{j}$ is a simple closed path in $G$. This can only occur if this simple path is $C_{n}$. If the path $p_{0}, \ldots, p_{2m+1}$ contains at least two copies of $C_{n}$, then by removing the edges of these two copies, we obtain a shorter even-connected path. Thus, the path $p_{0}, \ldots, p_{2m+1}$ contains exactly one copy of $C_{n}$. This shows that the even-connected path $p_{0}, \ldots, p_{2m+1}$ is of the form $x_{n+1},x_{1},x_{2}, \ldots, x_{n},x_{1},x_{n+1}$. Therefore, $2m+2=n+3$ and hence $n=2m-1$ is odd. By re-indexing if necessary, we may assume that $p_{i}=x_{i}$ for $i=1, \ldots, 2m-1$ and $p_{2m}=x_{1}$. Moreover, by the definition of even-connected path, we have
$$M=(p_{1}p_{2}) \ldots (p_{2m-1}p_{2m})N=(x_{1}x_{2}) \ldots (x_{2m-1}x_{1})N$$ where $N\in I^{s-m}$.

Using Theorem \ref{increase}, in order to prove the last statement of the lemma, we must show that for every $j=1, \ldots, 2n$, the vertices $x_{n+1}$ and $x_{j}$ are even-connected with respect to $M$. By assumptions, there is nothing to prove if $j=1$ or $n+1$. Hence assume that $j\neq 1 ,n+1$. We consider the following cases.

\vspace{0.2cm}
{\bf Case 1.}
If $1< j\leq n$ is odd, then $x_{n+1},x_{1}, \ldots, x_{j}$ is an even-connected path between $x_{n+1}$ and $x_{j}$.

\vspace{0.2cm}
{\bf Case 2.}
If $1< j\leq n$ is even, then $x_{j}, x_{j+1} \ldots, x_{1},x_{n+1}$ is an even-connected path between $x_{j}$ and $x_{n}$.

\vspace{0.2cm}
{\bf Case 3.}
Assume that $n+1< j\leq 2n$.

If $j$ is odd, then $x_{n+1},x_{1},x_{2},x_{3}, \ldots, x_{j-n},x_{j}$ is an even-connected path between $x_{n+1}$ and $x_{j}$.

If $j$ is even, then $x_{n+1},x_{1},x_{n},x_{n-1},x_{n-2}, \ldots, x_{j-n},x_{j}$ is an even-connected path between $x_{n+1}$ and $x_{j}$.

Therefore, $x_{n+1}x_{j}\in(I^{s+1}:M)$ for all $j=1, \ldots, 2n$.
\end{proof}

We are now ready to prove the main result of this paper.
\begin{thm} \label{mmmain}
Let $G=W(C_{n})$ be a whiskered cycle graph and $I=I(G)$ be its edge ideal. Then for all $s\geq 1$, we have
$${\rm reg}(I^{s})=2s+\lceil\frac{n-1}{2}\rceil-1= 2s+{\rm indmatch}(G)-1.$$
\end{thm}

\begin{proof}
One can easily check that ${\rm indmatch}(G)=\lceil\frac{n-1}{2}\rceil$. Thus, the second inequality is obvious. To prove the first inequality, note that by \cite[Theorem 4.5]{bht}, the inequality $${\rm reg}(I^{s})\geq 2s+{\rm indmatch}(G)-1$$is known. Therefore, we must prove that for every $s\geq 1$ the inequality $${\rm reg}(I^{s})\leq 2s+\lceil\frac{n-1}{2}\rceil-1$$holds. For $s=1$ the above inequality is known (see Proposition \ref{first}). By applying \cite[Theoem 5.2]{b} and using induction on $s$, it is enough to prove that for every $s\geq 1$ and every minimal generator $M$ of $I^{s}$, $${\rm reg}(I^{s+1}:M)\leq\lceil\frac{n-1}{2}\rceil+1.$$
By Theorem \ref{increase}, the ideal $(I^{s+1}:M)$ is generated by the quadratics $uv$, where it is either an edge ideal of $G$ or $u$ and $v$ are even-connected with respect to $M$. Let $J$ denote the polarization of the ideal $(I^{s+1}:M)$. Assume that $x_{i_{1}}^{2}, \ldots, x_{i_{t}}^{2}$  are the non-squarefree minimal generators of $(I^{s+1}:M)$. Then $$J=I(G')+(x_{i_{1}}y_{i_{1}}, \ldots, x_{i_{t}}y_{i_{t}}),$$
where $G'$ is a graph over the vertices ${x_{1}, \ldots, x_{2n}}$ and $y_{i_{1}}, \ldots, y_{i_{t}}$ are new variables. Since polarization does not change the regularity, we have ${\rm reg}(J)={\rm reg}(I^{s+1}:M)$.

On the other hand, since $(I^{s+1}:M)$ has all edges of $G$ as minimal generators, $G$ is a subgraph of $G'$. For every $j=0, \ldots, t$, let $H_{j}$ be the graph whose edge ideal is $$I(G')+(x_{i_{1}}y_{i_{1}}, \ldots, x_{i_{j}}y_{i_{j}}).$$
Then, $H_{0}=G'$ and $J=I(H_{t})$.
By Lemmas \ref{strong} and \ref{stable}, we observe that $\{x_{i_{j}},x_{l}\}\in E(G')$ for every $j=1, \ldots, t$ and every $l=1, \ldots, 2n$ with $i_{j}\neq l$. This implies that for every $1\leq j\leq t$, the graph $H_{j}\setminus N_{H_{j}}[x_{i_{j}}]$ consists of isolated vertices $\{y_{i_{1}}, \ldots, y_{i_{j-1}}\}$.
Hence, $${\rm reg}(H_{j}\setminus N_{H_{j}}[x_{i_{j}}])=0.$$
Now, by \cite[Lemma 2.10]{dhs} we have $${\rm reg}(H_{j})={\rm reg}(H_{j}\setminus x_{i_{j}}).$$
On the other hand, $y_{i_{j}}$ is an isolated vertex in $H_{j}\setminus x_{i_{j}}$ and $H_{j}\setminus \{x_{i_{j}},y_{i_{j}}\}$ is an induced subgraph of $H_{j}\setminus y_{i_{j}}=H_{j-1}$. Hence, $${\rm reg}(H_{j})={\rm reg}(H_{j}\setminus x_{i_{j}})={\rm reg}(H_{j}\setminus \{x_{i_{j}},y_{i_{j}}\})\leq {\rm reg}(H_{j}\setminus y_{i_{j}})={\rm reg}(H_{j-1}).$$
Note that $H_{j-1}$ is an induced subgraph of $H_{j}$ and this implies that ${\rm reg}(H_{j-1})\leq {\rm reg}(H_{j})$. Therefore, we obtain that ${\rm reg}(H_{j})= {\rm reg}(H_{j-1})$
for all $j=1, \ldots, t$.

This, in particular, implies that $${\rm reg}(J)={\rm reg}(H_{t})={\rm reg}(H_{0})={\rm reg}(G').$$
To complete the proof, it is enough to show that ${\rm reg}(G')\leq \lceil\frac{n-1}{2}\rceil+1$.

If $\{x_{n+i},x_{n+j}\}\in E(G')$, for some $1\leq i,j\leq n$, then $x_{n+i}x_{n+j}\in (I^{s+1}:M)$. This means that there exists an even-connected path $x_{n+i}=p_{0},p_{1}, \ldots, p_{2m+1}=x_{n+j}$ in $G$, between $x_{n+i}$ and $x_{n+j}$. Since $N_{G}(x_{n+i})=\{x_{i}\}$, we conclude that $p_{1}=x_{i}$. Thus, $x_{i+1}, p_{1}, \ldots, p_{2m+1}=x_{n+j}$ is an even-connected path between $x_{i+1}$ and $x_{n+j}$. Also, $x_{i-1}, p_{1}, \ldots, p_{2m+1}=x_{n+j}$ is an even-connected path between $x_{i-1}$ and $x_{n+j}$. Similarly, one can show that $x_{n+i}$ is even-connected to both $x_{j-1}$ and $x_{j+1}$ (where we consider the indices $i-1$, $i+1$, $j-1$ and $j+1$ modulo $n$). This shows that$$\{x_{i-1},x_{n+j}\}, \{x_{i+1},x_{n+j}\}, \{x_{j-1},x_{n+i}\}, \{x_{j+1},x_{n+i}\}\in E(G').$$It follows from Lemma \ref{main} that ${\rm reg}(G')\leq\lceil\frac{n+1}{2}\rceil$ and this completes the proof.
\end{proof}


\section{proof of lemma 2.1}

In this section, we prove Lemma \ref{choose}.
\begin{proof}
We use induction on $n$. Since every connected graph with at most four vertices has chordal complement, the result is true for $n\leq 2$. Set$$d={\rm deg}_{G}(x_{n+1})+{\rm deg}_{G}(x_{n+2}).$$We prove the assertion by induction on $d$. Note that $d\geq 2$.

We first consider the case $d=2$. In this case, let $H=G\setminus \{x_{1},x_{n+1}\}$ and note that in the graph $G\setminus x_{1}$, the vertex $x_{n+1}$ is an isolated vertex. Thus, ${\rm reg}(G\setminus x_{1})={\rm reg}(H)$. Hence, by the induction hypothesis on $n$, we get
$${\rm reg}(G\setminus x_{1})={\rm reg}(H)\leq \lceil\frac{(n-1)+2}{2}\rceil\leq\lceil\frac{n+2}{2}\rceil. \ \ \ \ \ \ \ \ \ \ \ \ \ \ (3.1)$$
Let $K$ be the induced subgraph of $G$ over the vertices $V(G)\setminus \{x_{1},x_{n+1},x_{2},x_{n+2}\}$. Clearly, $G\setminus(N[x_{1}]\cup \{x_{n+2}\})$ is an induced subgraph of $K$. So, by Theorem \ref{change}, $${\rm reg}(G\setminus(N[x_{1}]\cup \{x_{n+2}\}))\leq {\rm reg}(K).$$
On the other hand, ${\rm deg}_{G}(x_{n+2})=1$ and $x_{n+2}$ is an isolated vertex in $G\setminus N[x_{1}]$. Thus, $${\rm reg}(G\setminus N[x_{1}])={\rm reg}(G\setminus (N[x_{1}]\cup \{x_{n+2}\})),$$ and hence, ${\rm reg}(G\setminus N[x_{1}])\leq {\rm reg}(K)$.
Moreover, by induction on $n$, we have
$${\rm reg}(K)\leq \lceil\frac{(n-2)+2}{2}\rceil=\lceil\frac{n}{2}\rceil.\ \ \ \ \ \ \ \ \ \ \ \ \ \ (3.2)$$
It then follows from Theorem \ref{change}, together with $(3.1)$ and $(3.2)$ that $${\rm reg}(G)\leq\max\{{\rm reg}(G\setminus x_{1}), {\rm reg}(G\setminus N[x_{1}])+1\}\leq\lceil\frac{n+2}{2}\rceil.$$
Now assume that $d>2$. In this case, there is a vertex $x_{t}\in \{x_{1}, \ldots, x_{n},x_{n+1}, \ldots, x_{2n}\}$ such that either

\begin{itemize}
\item[(i)] $t\neq1$ and $\{x_{n+1},x_{t}\}\in E(G)$ or

\item[(ii)] $t\neq2$ and $\{x_{n+2},x_{t}\}\in E(G)$.
\end{itemize}
Let $t$ be the smallest integer with this property. We consider the following cases.

\vspace{0.2cm}
{\bf Case 1.}
Assume that $e=\{x_{n+1},x_{n+2}\}\in E(G)$. Then $${\rm deg}_{G\setminus e}(x_{n+1})+{\rm deg}_{G\setminus e}(x_{n+2})<{\rm deg}_{G}(x_{n+1})+{\rm deg}_{G}(x_{n+2}).$$
So, by induction hypothesis on $d$, we have $${\rm reg}(G\setminus e)\leq\lceil\frac{n+2}{2}\rceil.$$
Let $H$ be the induced subgraph of $G$ over the vertices $V(G)\setminus \{x_{1},x_{n+1},x_{2},x_{n+2}\}$. Then, by induction on $n$, $${\rm reg}(H)\leq \lceil\frac{(n-2)+2}{2}\rceil=\lceil\frac{n}{2}\rceil.$$
Moreover, $G_{e}$ is an induced subgraph of $H$ and it follows from Theorem \ref{change} that ${\rm reg}(G_{e})\leq {\rm reg}(H)\leq\lceil\frac{n}{2}\rceil$. Hence, by Theorem \ref{change}, we conclude that $${\rm reg}(G)\leq\lceil\frac{n+2}{2}\rceil.$$

\vspace{0.2cm}
{\bf Case 2.}
Assume that $e=\{x_{n+1},x_{2}\}\in E(G)$ and $\{x_{n+1},x_{n+3}\}\notin E(G)$. Then $${\rm deg}_{G\setminus e}(x_{n+1})+{\rm deg}_{G\setminus e}(x_{n+2})<{\rm deg}_{G}(x_{n+1})+{\rm deg}_{G}(x_{n+2}),$$
and since $\{x_{n+1},x_{n+3}\}\notin E(G)$, it follows that $G\setminus e$ satisfies the assumptions of the lemma. So, by induction hypothesis on $d$, we have $${\rm reg}(G\setminus e)\leq\lceil\frac{n+2}{2}\rceil.$$
Let $H$ be the induced subgraph of $G$ over the vertices $V(G)\setminus \{x_{1},x_{n+1},x_{2},x_{n+2}\}$. Then, by induction on $n$, $${\rm reg}(H)\leq \lceil\frac{(n-2)+2}{2}\rceil=\lceil\frac{n}{2}\rceil.$$
Moreover, $G_{e}$ is an induced subgraph of $H$ and it follows from Theorem \ref{change} that ${\rm reg}(G_{e})\leq {\rm reg}(H)\leq\lceil\frac{n}{2}\rceil$. Hence, by Theorem \ref{change}, we conclude that $${\rm reg}(G)\leq\lceil\frac{n+2}{2}\rceil.$$

\vspace{0.2cm}
{\bf Case 3.}
Assume that $e=\{x_{n+1},x_{2}\}$ and $e'=\{x_{n+1},x_{n+3}\}$ are edges of $G$. Since in $G\setminus e'$ we have $${\rm deg}_{G\setminus e'}(x_{n+1})+{\rm deg}_{G\setminus e'}(x_{n+2})<{\rm deg}_{G}(x_{n+1})+{\rm deg}_{G}(x_{n+2}),$$
by induction on $d$, it follows that $${\rm reg}(G\setminus e')\leq\lceil\frac{n+2}{2}\rceil.$$ Using Theorem \ref{change}, to show that ${\rm reg}(G)\leq\lceil\frac{n+2}{2}\rceil$, it suffices to prove that ${\rm reg}(G_{e'})\leq\lceil\frac{n}{2}\rceil$.

Let $H$ be the induced subgraph of $G$ over the vertices $V(G)\setminus \{x_{1},x_{n+1},x_{3},x_{n+3}\}$ and let $H'$ be the graph with the same vertex set as $H$ and the edge set $$E(H')=E(H)\cup\big\{\{x_{2},x_{4}\}\big\}\cup\big\{\{x_{2},x_{n+i}\}\mid 4\leq i\leq n\big\}\cup\big \{\{x_{4},x_{n+i}\}\mid 2\leq i\leq n, i\neq 3\big\}.$$
Then, $H'$ satisfies the assumptions of the lemma. Hence, by induction hypothesis on $n$, we see that $${\rm reg}(H')\leq\lceil\frac{(n-2)+2}{2}\rceil=\lceil\frac{n}{2}\rceil.$$
Moreover, $G_{e'}$ is an induced subgraph of $H'$ (since $x_{2}$ and $x_{4}$ are not vertices of $G_{e'}$), which implies that $${\rm reg}(G_{e'})\leq {\rm reg}(H')\leq\lceil\frac{n}{2}\rceil.$$

\vspace{0.2cm}
{\bf Case 4.}
Assume that $e=\{x_{n+2},x_{1}\}$ is an edge of $G$. Then $${\rm deg}_{G\setminus e}(x_{n+1})+{\rm deg}_{G\setminus e}(x_{n+2})<{\rm deg}_{G}(x_{n+1})+{\rm deg}_{G}(x_{n+2}).$$
Thus, by induction on $d$,  $${\rm reg}(G\setminus e)\leq\lceil\frac{n+2}{2}\rceil.$$
Let $H$ be the induced subgraph of $G$ over the vertices $V(G)\setminus \{x_{1},x_{n+1},x_{2},x_{n+2}\}$. Then, $G_{e}$ is an induced subgraph of $H$ and the induction hypothesis implies that $${\rm reg}(G_{e})\leq {\rm reg}(H)\leq\lceil\frac{n}{2}\rceil.$$
It follows from Theorem \ref{change} that $${\rm reg}(G)\leq\lceil\frac{n+2}{2}\rceil.$$

\vspace{0.2cm}
{\bf Case 5.} Assume that $3\leq t\leq n$ and $e=\{x_{n+1},x_{t}\}\in E(G)$ and $\{x_{n+1},x_{n+t+1}\}\notin E(G)$.

If $t=3$, then $e=\{x_{n+1},x_{3}\}\in E(G)$ and $\{x_{n+1},x_{n+4}\}\notin E(G)$. If $\{x_{n+1},x_{n+2}\}$ is an edge of $G$, then the assertion follows from case $1$.

If $t\geq4$ and $\{x_{n+1},x_{n+t-1}\}$ is an edge of $G$, then by assumption, $\{x_{n+1},x_{t-2}\}$ is an edge of $G$, which is a contradiction by the choice of $t$.

Thus, we assume that $\{x_{n+1},x_{n+t-1}\}\notin E(G)$. Since $\{x_{n+1},x_{n+t-1}\}$ and $\{x_{n+1},x_{n+t+1}\}$ are not edges of $G$, it follows that $G\setminus e$ satisfies the assumptions of the lemma. Now, $${\rm deg}_{G\setminus e}(x_{n+1})+{\rm deg}_{G\setminus e}(x_{n+2})<{\rm deg}_{G}(x_{n+1})+{\rm deg}_{G}(x_{n+2}).$$Thus, by induction on $d$, $${\rm reg}(G\setminus e)\leq\lceil\frac{n+2}{2}\rceil.$$
Let $H$ be the induced subgraph of $G$ over the vertices $V(G)\setminus \{x_{1},x_{n+1},x_{3},x_{n+3}\}$ and let $H'$ be the graph with the same vertex set as $H$ and the edge set $$E(H')=E(H)\cup\big\{\{x_{2},x_{4}\}\big\}\cup\big\{\{x_{2},x_{n+i}\}\mid 4\leq i\leq n\big\}\cup\big\{\{x_{4},x_{n+i}\}\mid 2\leq i\leq n, i\neq 3\big\}.$$
Then $G_{e}$ is an induced subgraph of $H'$ (since $x_{2}$ and $x_{4}$ are not vertices of $G_{e}$). Therefore, by induction on $n$ and Theorem \ref{change}, we have $${\rm reg}(G_{e})\leq {\rm reg}(H')\leq\lceil\frac{(n-2)+2}{2}\rceil=\lceil\frac{n}{2}\rceil.$$
The conclusion that ${\rm reg}(G)\leq\lceil\frac{n+2}{2}\rceil$ now follows from Theorem \ref{change}.

\vspace{0.2cm}
{\bf Case 6.}
Assume that $3\leq t\leq n$ and  $e=\{x_{n+1},x_{t}\}$ and suppose that $e'=\{x_{n+1},x_{n+t+1}\}$ are edges of $G$. Then $${\rm deg}_{G\setminus e'}(x_{n+1})+{\rm deg}_{G\setminus e'}(x_{n+2})<{\rm deg}_{G}(x_{n+1})+{\rm deg}_{G}(x_{n+2}),$$ and by induction on $d$, $${\rm reg}(G\setminus e')\leq\lceil\frac{n+2}{2}\rceil.$$
According to Theorem \ref{change}, to prove that ${\rm reg}(G)\leq\lceil\frac{n+2}{2}\rceil$, it remains to show that $${\rm reg}(G_{e'})\leq\lceil\frac{n}{2}\rceil.$$
Let $H$ be the induced subgraph of $G$ over the vertices $V(G)\setminus \{x_{1},x_{n+1},x_{t+1},x_{n+t+1}\}$ and let $H'$ be the graph with the same vertex set as $H$ and the edge set

\begin{align*}
E(H')=E(H) & \cup\big\{\{x_{t},x_{t+2}\big\}\}\cup\big\{\{x_{t},x_{n+i}\}\mid 2\leq i\leq n, i\neq t+1\big\}\\
 & \cup\big\{\{x_{t+2},x_{n+i}\}\mid 2\leq i\leq n, i\neq t+1\big\}.
\end{align*}

One can easily check that $H'$ satisfies the assumptions of the lemma. Hence by induction on $n$, we conclude that  $${\rm reg}(H')\leq\lceil\frac{(n-2)+2}{2}\rceil=\lceil\frac{n}{2}\rceil.$$
Moreover, $G_{e'}$ is an induced subgraph of $H'$ (since $x_{t}$ and $x_{t+2}$ are not vertices of $G_{e'}$). This implies that ${\rm reg}(G_{e'})\leq {\rm reg}(H')\leq\lceil\frac{n}{2}\rceil$ and the result follows.

\vspace{0.2cm}
{\bf Case 7.}
Assume that $n+3\leq t\leq 2n$ and  $\{x_{n+1},x_{t}\}$ is an edge of $G$. Then, by assumption,  $\{x_{n+1},x_{t-n-1}\}$ is an edge of $G$, which is contradiction by the choice of $t$.

\vspace{0.2cm}
{\bf Case 8.}
Assume that $3\leq t\leq n$ and suppose that $e=\{x_{n+2},x_{t}\}$ and $e'=\{x_{n+2},x_{n+t+1}\}$ are edges of $G$. In the graph $G\setminus e'$, we have $${\rm deg}_{G\setminus e'}(x_{n+1})+{\rm deg}_{G\setminus e'}(x_{n+2})<{\rm deg}_{G}(x_{n+1})+{\rm deg}_{G}(x_{n+2}).$$ So, by induction on $d$, $${\rm reg}(G\setminus e')\leq\lceil\frac{n+2}{2}\rceil.$$
It follows from Theorem \ref{change} that $${\rm reg}(G)\leq \max\{{\rm reg}(G\setminus e'), {\rm reg}(G_{e'})+1\}.$$
Thus, in order to prove that ${\rm reg}(G)\leq\lceil\frac{n+2}{2}\rceil$, it is enough to show that $${\rm reg}(G_{e'})\leq\lceil\frac{n}{2}\rceil. \ \ \ \ \ \ \ \ \ \ \ \ \ \ \ \ \ \ (3.3)$$
Let $H$ be the induced subgraph of $G$ over the vertices $V(G)\setminus \{x_{2},x_{n+2},x_{t+1},x_{n+t+1}\}$ and let $H'$ be the graph with the same vertex set as $H$ and the edge set
\begin{align*}
& E(H') =E(H) \cup\big\{\{x_{1},x_{3}\}, \{x_{t},x_{t+2}\}\big\}\cup\big\{\{x_{1},x_{n+i}\}\mid 3\leq i\leq n, i\neq t+1\big\}\\ & \cup\big\{\{x_{3},x_{n+i}\}\mid 1\leq i\leq n, i\neq 2, t+1\big\}\cup
\big\{\{x_{t},x_{n+i}\}\mid 1\leq i\leq n, i\neq 2, t+1\big\}\\ & \cup\big\{\{x_{t+2},x_{n+i}\}\mid 1\leq i\leq n, i\neq 2, t+1\big\}.
\end{align*}
One can easily check that $H'$ satisfies the assumptions of the lemma. Hence by induction on $n$, $${\rm reg}(H')\leq\lceil\frac{(n-2)+2}{2}\rceil=\lceil\frac{n}{2}\rceil.$$
Moreover, $G_{e'}$ is an induced subgraph of $H'$ (since $x_1, x_3, x_{t}$ and $x_{t+2}$ are not vertices of $G_{e'}$). So, by Theorem \ref{change},$${\rm reg}(G_{e'})\leq {\rm reg}(H')\leq\lceil\frac{n}{2}\rceil$$ and $(3.3)$ holds.

\vspace{0.2cm}
{\bf Case 9.}
Assume that $e=\{x_{n+2},x_{n+3}\}$ is an edge of $G$. By induction on $d$, $${\rm reg}(G\setminus e)\leq\lceil\frac{n+2}{2}\rceil.$$
Now, using Theorem \ref{change} $${\rm reg}(G)\leq \max\{{\rm reg}(G\setminus e), {\rm reg}(G_{e})+1\}.$$
Thus, it is enough to show that $${\rm reg}(G_{e})\leq\lceil\frac{n}{2}\rceil.$$
Let $H$ be the induced subgraph of $G$ over the vertices $V(G)\setminus \{x_{2},x_{n+2},x_{3},x_{n+3}\}$ and let $H'$ be the graph with the same vertex set as $H$ and the edge set $$E(H')=E(H)\cup\big\{\{x_{1},x_{4}\}\big\}\cup\big\{\{x_{1},x_{n+i}\}\mid 4\leq i\leq n\big\}\cup\big\{\{x_{4},x_{n+i}\}\mid 1\leq i\leq n, i\neq 2.3\big\}.$$
One can easily check that $H'$ satisfies the assumptions of the lemma. Therefore, by induction on $n$, we have $${\rm reg}(H')\leq\lceil\frac{(n-2)+2}{2}\rceil=\lceil\frac{n}{2}\rceil.$$
Moreover, $G_{e}$ is an induced subgraph of $H'$ (since $x_1$ and $x_4$ are not vertices of $G_{e}$) which implies that$${\rm reg}(G_{e})\leq {\rm reg}(H')\leq\lceil\frac{n}{2}\rceil.$$

\vspace{0.2cm}
{\bf Case 10.}
Assume that $e=\{x_{n+2},x_{3}\}$ is an edge of $G$ and $\{x_{n+2},x_{n+4}\}$ is not an edge of $G$. Note that$${\rm deg}_{G\setminus e}(x_{n+1})+{\rm deg}_{G\setminus e}(x_{n+2})<{\rm deg}_{G}(x_{n+1})+{\rm deg}_{G}(x_{n+2}).$$
Since $\{x_{n+2}, x_{n+4}\}$ is not an edge of $G$, we conclude that $G\setminus e$ satisfies the assumptions of the lemma. So, by induction on $d$,$${\rm reg}(G\setminus e)\leq\lceil\frac{n+2}{2}\rceil.$$Let $H$ be the induced subgraph of $G$ over the vertices $V(G)\setminus \{x_{2},x_{n+2},x_{3},x_{n+3}\}$ and let $H'$ be the graph with the same vertex set as $H$ and the edge set
\begin{align*}
& E(H')=E(H)\cup\big\{\{x_{1},x_{4}\}\big\}\cup\big\{\{x_{4},x_{n+i}\}\mid 1\leq i\leq n, i\neq 2, 3\big\}\\ & \cup \big\{\{x_{1},x_{n+i}\}\mid \{x_{n+4},x_{n+i}\}\in E(G), i\neq 2, 3\big\}.
\end{align*}
One can easily check that $H'$ satisfies the assumptions of the lemma.
Note that $x_4$ is not a vertex of $G_{e}$. On the other hand, if $\{x_{n+4},x_{n+i}\}\in E(G)$, for some $i\neq 2, 3$, then the assumptions of the lemma implies that $\{x_{3},x_{n+i}\}\in E(G)$. Consequently, $x_{n+i}$ is not a vertex of $G_{e}$. Therefore, $G_{e}$ is an induced subgraph of $H'$. Hence, by Theorem \ref{change} and by induction on $n$, $${\rm reg}(G_{e})\leq {\rm reg}(H')\leq\lceil\frac{(n-2)+2}{2}\rceil=\lceil\frac{n}{2}\rceil.$$
Finally, using Theorem \ref{change}, we have, $${\rm reg}(G)\leq\lceil\frac{n+2}{2}\rceil.$$

\vspace{0.2cm}
{\bf Case 11.}
Assume that $4\leq t\leq n$ and suppose that $e=\{x_{n+2},x_{t}\}$ is an edge of $G$ and $\{x_{n+2},x_{n+t+1}\}$ is not an edge of $G$.

If $t=4$ and $\{x_{n+2},x_{n+3}\}\in E(G)$, then the assertion follows from case 9.

If $t\geq5$ and $\{x_{n+2},x_{n+t-1}\}\in E(G)$, then by assumption, $\{x_{n+2},x_{t-2}\}\in E(G)$, which is a contradiction by the choice of $t$.

Thus, we assume that $\{x_{n+2}, x_{n+t-1}\}\notin E(G)$. In this case, we observe that $${\rm deg}_{G\setminus e}(x_{n+1})+{\rm deg}_{G\setminus e}(x_{n+2})<{\rm deg}_{G}(x_{n+1})+{\rm deg}_{G}(x_{n+2}).$$
Since $\{x_{n+2}, x_{n+t-1}\}$ and $\{x_{n+2},x_{n+t+1}\}$ are not edges of $G$, we conclude that $G\setminus e$ satisfies the assumptions of the lemma. So, by induction on $d$,$${\rm reg}(G\setminus e)\leq\lceil\frac{n+2}{2}\rceil.$$
Let $H$ be the induced subgraph of $G$ over the vertices $V(G)\setminus \{x_{2},x_{n+2},x_{t},x_{n+t}\}$ and let $H'$ be the graph with the same vertex set as $H$ and the edge set
\begin{align*}
& E(H') =E(H)\cup\big\{\{x_{1},x_{t-1}\}, \{x_{3},x_{t+1}\}\big\}\cup\big\{\{x_{t-1},x_{n+i}\}\mid 1\leq i\leq n, i\neq 2, t\big\} \\ & \cup\big\{\{x_{t+1},x_{n+i}\}\mid 1\leq i\leq n, i\neq 2, t\big\} \cup
\big\{\{x_{1},x_{n+i}\}\mid \{x_{n+t-1},x_{n+i}\}\in E(G), i\neq 2, t\big\} \\ & \cup
\big\{\{x_{3},x_{n+i}\}\mid \{x_{n+t+1},x_{n+i}\}\in E(G), i\neq 2, t\big\}.
\end{align*}
By considering the path$$x_1, x_{t-1}, x_{t-2}, \ldots, x_4, x_3, x_{t+1}, x_{t+2}, \ldots, x_n,$$one can easily check that $H'$ satisfies the assumptions of the lemma.
Note that $x_{t-1}$ and $x_{t+1}$ are not vertices of $G_{e}$. On the other hand, if $\{x_{n+t-1},x_{n+i}\}\in E(G)$, for some $i\neq 2, t$, then the assumptions of the lemma implies that $\{x_{t},x_{n+i}\}\in E(G)$. Consequently, $x_{n+i}$ is not a vertex of $G_{e}$. Similarly, if $\{x_{n+t+1},x_{n+i}\}\in E(G)$, for some $i\neq 2, t$, then $x_{n+i}$ is not a vertex of $G_{e}$. Therefore, $G_{e}$ is an induced subgraph of $H'$. Hence, by Theorem \ref{change} and by induction on $n$, $${\rm reg}(G_{e})\leq {\rm reg}(H')\leq\lceil\frac{(n-2)+2}{2}\rceil=\lceil\frac{n}{2}\rceil.$$
Finally, using Theorem \ref{change}, we have, $${\rm reg}(G)\leq\lceil\frac{n+2}{2}\rceil.$$

\vspace{0.2cm}
{\bf Case 12.}
Assume that $n+4\leq t\leq 2n$ and $e=\{x_{n+2},x_{t}\}$ is an edge of $G$. Then by assumption, $\{x_{n+2},x_{t-n-1}\}\in E(G)$, which is a contradiction by the choice of $t$.
\end{proof}



\end{document}